\documentclass[11pt,a4paper,draft]{article}

\usepackage[ngerman,english]{babel}
\usepackage[utf8]{inputenc}

\usepackage[DIV=14]{typearea}

\usepackage[affil-it]{authblk}

\usepackage{xparse}

\usepackage{amsmath, amsfonts, amssymb, amsthm}
\usepackage{mathrsfs}
\usepackage{mathtools}

\usepackage[noadjust]{cite}

\usepackage[capitalise]{cleveref}

\crefname{equation}{}{}

\newcommand{\R}{\mathbb{R}}

\newcommand{\N}{\mathbb{N}}

\newcommand{\ee}{\mathrm{e}}

\DeclareDocumentCommand\dd{ o g d() }{
	\IfNoValueTF{#2}{
		\IfNoValueTF{#3}
			{\mathrm{d}\IfNoValueTF{#1}{}{^{#1}}}
			{\mathinner{\mathrm{d}\IfNoValueTF{#1}{}{^{#1}}\argopen(#3\argclose)}}
		}
		{\mathinner{\mathrm{d}\IfNoValueTF{#1}{}{^{#1}}#2} \IfNoValueTF{#3}{}{(#3)}}
	}

\newcommand{\del}{\partial}
\newcommand{\eps}{\varepsilon}
\newcommand{\vcc}{\vcentcolon}

\DeclarePairedDelimiter\abs{\lvert}{\rvert}
\DeclarePairedDelimiter\norm{\Vert}{\rVert}

\theoremstyle{plain}
\newtheorem{theorem}{Theorem}[section]
\newtheorem{lemma}[theorem]{Lemma}
\newtheorem{proposition}[theorem]{Proposition}

\theoremstyle{definition}

\theoremstyle{remark}
\newtheorem{remark}[theorem]{Remark}

\title{Continuum limit for interacting systems on adaptive networks}
\author{Sebastian Throm \thanks{\texttt{sebastian.throm@umu.se}}}
  
 \affil{\em Ume{\aa} University,  Department of Mathematics and Mathematical Statistics, 901 87 Umeå }
\date{}

\begin{document}

\maketitle

\begin{abstract}
 The article considers systems of interacting particles on networks with adaptively coupled dynamics. Such processes appear frequently in natural processes and applications. Relying on the notion of graph convergence, we prove that for large systems the dynamics can be approximated by the corresponding continuum limit. Well-posedness of the latter is also established. 
\end{abstract}

\section{Introduction}

Models of collectively interacting particles play a crucial role in many branches of the natural sciences including biological systems, industrial processes and social activities \cite{JaK01,SYT02,CFL09,CFT10,WHK22}. Many of these real-world examples exhibit an underlying network structure and consequently, there has been an increasing interest during the last years in corresponding mathematical models \cite{BBV08,PoG16,AyP21,WHK22,Bur22,NGW23}. In the case of finite systems, i.e.\@ where finitely many particles interact, this typically leads to a large system of coupled ordinary differential equations (ODEs):
\begin{equation}\label{eq:intro:1}
 \dot{\phi}_{k}=\frac{1}{N}\sum_{\ell=1}^{N}\kappa_{k\ell}g(\phi_{k},\phi_{\ell}) \qquad \text{with } k=1,\ldots,N.
\end{equation}
Here $\phi_k(t)$ describes the state of the $k$'s particle at time $t$, the function $g$ models the interaction between two particles and $\kappa_{k\ell}$ corresponds to the adjacency matrix of the underlying network. More precisely, each particle is assumed to be located at the node of a graph consisting of $N$ nodes which are labeled by $1,\ldots,N$. The quantity $\kappa_{k\ell}$ denotes the weight of the edge between the nodes $k$ and $\ell$. One of the most prominent examples is the classical \emph{Kuramoto model} where $g(\phi_k,\phi_\ell)=\sin(\phi_\ell-\phi_k)$ and $\kappa_{k\ell}\equiv\kappa$ \cite{Kur84}.

In many applications the number $N$ of involved particles is so large that the evolution of the whole system is not tractable. Instead, one is interested in continuous limiting descriptions when $N\to\infty$ and, for systems without a network structure, i.e.\@ $\kappa_{k\ell}\equiv const$, there is a well established theory available \cite{Gol16}. Moreover, in recent years, based on the notion of graph convergence, it has been possible to extend these methods to situations with an underlying network \cite{Med14,KaM17,KaM18,KuT19}. More precisely, assuming that the (stationary) graph structure has for $N\to\infty$ a suitable limiting \emph{graphon}, corresponding continuum and mean-field models have been derived.

\subsection{Coupled oscillators on adaptive networks}

However, for many systems, the network structure is not fixed but instead, it evolves in time while this evolution is often also coupled to the particle dynamics. A special case is given by the following adaptively coupled Kuramoto model considered in \cite{BFK19}:
\begin{equation}\label{eq:Kuramoto:adaptive}
 \begin{split}
  \dot{\phi}_{k}&=\omega-\frac{1}{N}\sum_{\ell=1}^{N}\kappa_{k\ell}\sin(\phi_{\ell}-\phi_{k}+\alpha)\\
  \dot{\kappa}_{k\ell}&=-\eps(\sin(\phi_{k}-\phi_{\ell}+\beta)+\kappa_{k\ell}).
 \end{split}
\end{equation}
Moreover, in \cite{HNP16} the synchronisation of oscillators following the slightly generalised model
\begin{equation}\label{eq:dyn:model:1}
 \begin{split}
  \dot{\phi}_{k}&=\frac{1}{N}\sum_{\ell=1}^{N}\kappa_{k\ell}g(t,\phi_{\ell}-\phi_{k})+\omega_{k}\\
  \dot{\kappa}_{k\ell}&=\Gamma(\phi_{\ell}-\phi_{k})-\gamma\kappa_{k\ell}.
 \end{split}
\end{equation}
has been considered. By means of Duhamel's formula one can solve the second equation explicitly which yields $\kappa_{k\ell}(t)=\kappa_{k\ell}(0)\ee^{-\gamma t}+\int_{0}^{t}\Gamma(\phi_{\ell}(s)-\phi_{k}(s))\ee^{-\gamma(t-s)}\dd{s}$. Plugging this expression back into the first equation reduces the problem again to an equation with a stationary network up to an additional time integration. In \cite{GKX21} the continuum limit has been derived for this kind of graph dynamics.
\subsection{A generalised model}
In this work, we will consider the following generalised model where the evolution of $\kappa_{k\ell}$ does not only depend on its current state and the dynamics of $\phi_k$ and $\phi_\ell$ but instead, it might be influenced by the whole system. Moreover, we allow each edge/weight of the network to follow its own dynamics. In fact, we will study the model
\begin{equation}\label{eq:gen:model}
 \begin{split}
  \dot{\phi}_{k}&=\frac{1}{N}\sum_{\ell=1}^{N}\kappa_{k\ell}g(t,\phi_{k},\phi_{\ell})+f_{k}(t,\phi)\\
  \dot{\kappa}_{k\ell}&=\Lambda_{k\ell}(t,\kappa,\phi).
 \end{split}
\end{equation}
with continuous functions $f_{k}\colon [0,\infty)\times (\R^{d})^{N}\to \R^{d}$, $g\colon [0,\infty)\times (\R^{d})^2\to \R^{d}$ and $\Lambda_{k\ell}\colon [0,\infty)\times \R^{N\times N}\times(\R^d)^{N}\to \R$ whose properties will be specified more closely later and $\phi=(\phi_1,\ldots,\phi_N)$ as well as $\kappa=(\kappa_{k\ell})_{k,\ell=1}^{N}$.

\subsection{Assumptions and main result}

In order to derive the continuum limit, we follow the same approach developed e.g.\@ in \cite{Med14} which has also been exploited in \cite{AyP21}. For this aim, we parametrise the discrete system and the underlying graph over the sets $I=[0,1)$ and $I\times I=[0,1)\times [0,1)$ respectively. Precisely, denoting $I_{k}=[(k-1)/N,k/N)$ we set
\begin{equation}\label{eq:parametrisation}
 \begin{split}
  u^N(t,x)&\vcc=\sum_{k=1}^{N}\phi_{k}(t)\chi_{I_{k}}(x)\\
  K^{N}(t,x,y)&\vcc=\sum_{k,\ell=1}^{N}\kappa_{k\ell}(t)\chi_{I_{k}}(x)\chi_{I_{\ell}}(y)
 \end{split}
\end{equation}
where $\chi_{I_k}$ is the characteristic function of $I_k$. Moreover, given $\Lambda\colon [0,\infty)\times I\times I \times L^{\infty}(I\times I,\R) \times L^{\infty}(I,\R^d)\to \R$ and $f\colon[0,\infty)\times I\times L^{\infty}(I,\R^{d})\to \R^{d}$ satisfying the properties \cref{eq:Ass:Lambda,eq:Ass:f:g} below, we can reconstruct a corresponding discrete system via
\begin{equation}\label{eq:def:Lambda:discrete}
 \begin{aligned}
   \Lambda_{k\ell}(t,\kappa,\phi)&\vcc=N^2\int_{I_{k}\times I_{\ell}}\Lambda(t,x,y,K^{N}(t,\cdot,\cdot),u^{N}(t,\cdot))\dd{x}\dd{y}\\
   f_{k}(t,\phi)&\vcc= N\int_{I_{k}}f(t,x,u^{N}(t,\cdot))\dd{x}.
 \end{aligned}
\end{equation}
With this notation, \eqref{eq:gen:model} can be rewritten as the following integral equation
\begin{equation}\label{eq:def:u}
 \begin{split}
   \del_{t}u^{N}(t,x)&=\int_{I}K^{N}(t,x,y)g(t,u^{N}(t,x),u^{N}(t,y))\dd{y}+N\int_{\lfloor Nx\rfloor/N}^{(\lfloor Nx\rfloor+1)/N}f(t,\xi,u^{N}(t,\cdot))\dd{\xi}\\
   \del_{t}K^{N}(t,x,y)&=N^2\int_{\lfloor Nx\rfloor/N}^{(\lfloor Nx\rfloor+1)/N}\int_{\lfloor Ny\rfloor/N}^{(\lfloor Ny\rfloor+1)/N}\Lambda(t,\xi,\eta,K^{N}(t,\cdot,\cdot),u^{N}(t,\cdot))\dd{\xi}\dd{\eta}.
 \end{split}
\end{equation}
We assume that $\Lambda\colon [0,\infty)\times I\times I \times L^{\infty}(I\times I,\R) \times L^{\infty}(I,\R^d)\to \R$ is continuous and satisfies
\begin{equation}\label{eq:Ass:Lambda}
 \begin{aligned}
 \norm{\Lambda(t,\cdot,\cdot,K_1,u_1)-\Lambda(t,\cdot,\cdot,K_2,u_2)}_{L^2(I^2)}&\leq L_{\Lambda}\bigl(\norm{K_1-K_2}_{L^2(I^2)}+\norm{u_1-u_2}_{L^2(I)}\bigr)\\
  \abs{\Lambda(t,x,y,K,u)}&\leq B_{\Lambda}(1+\norm{K}_{L^{\infty}(I^2)}).
  \end{aligned}
\end{equation}
Moreover, we assume that $g\colon[0,T]\times (\R^{d})^2\to \R^{d}$ and $f\colon [0,T]\times I\times L^{\infty}(I,\R^{d})\to \R^{d}$ are continuous and satisfy the following estimates for all $\xi,\xi_1,\xi_2,\eta,\eta_1,\eta_2\in\R^d$ and  $u,u_1,u_2\in L^{\infty}(I,\R^{d})$ uniformly in $t$:
\begin{equation}\label{eq:Ass:f:g}
 \begin{aligned}
  \abs{f(t,\cdot,u)}&\leq B_{f}(1+\norm{u}_{L^{\infty}}) &&\text{and} &  \norm{f(t,\cdot,u_1)-f(t,\cdot,u_2)}_{L^{2}(I)}&\leq L_{f}\norm{u_1-u_2}_{L^{2}(I)}\\
  \abs{g(t,\xi,\eta}&\leq B_{g} && \text{and} & \abs{g(t,\xi_1,\eta_1)-g(t,\xi_2,\eta_2)}&\leq L_{g}(\abs{\xi_1-\xi_2}+\abs{\eta_1-\eta_2}).
 \end{aligned}
\end{equation}

Our main result in this work is the following theorem which states that in the limit of infinitely many particles, the discrete system \eqref{eq:gen:model} can be approximated by the integro-differential equation~\eqref{eq:gen:cont:limit}.

\begin{theorem}\label{Thm:cont:limit}
 Let $f,g\colon [0,T]\times \R^d\to \R^d$ satisfy \eqref{eq:Ass:f:g} and let $\Lambda\colon [0,\infty)\times I\times I \times L^{\infty}(I\times I,\R) \times L^{\infty}(I,\R^d)\to \R$ satisfy \eqref{eq:Ass:Lambda}. Assume that $K^{N}(0,\cdot,\cdot)$ has a limiting graphon $W$ with respect to $\norm{\cdot}_{L^{2}}$ which is uniformly bounded, i.e.\@ $\lim_{N\to\infty}\norm{K^{N}(0,\cdot,\cdot)-W}_{L^{2}(I\times I)}=0$ and $\norm{W}_{L^{\infty}(I\times I)}<\infty$. Then, as $N\to\infty$, the parametrisation $(u^{N},K^{N})$ given in \eqref{eq:def:u} which corresponds to the discrete system~\eqref{eq:gen:model} with \eqref{eq:def:Lambda:discrete} converges to its continuum limit $(u,K)$, i.e.\@ the unique solution of
\begin{equation}\label{eq:gen:cont:limit}
 \begin{aligned}
      \del_{t}u(t,x)&=\int_{I}K(t,x,y)g(t,u(t,x),u(t,y))\dd{y}+f(t,x,u(t,\cdot))\\
   \del_{t}K(t,x,y)&=\Lambda(t,x,y,K(t,\cdot,\cdot),u(t,\cdot))
 \end{aligned}
\end{equation}
with $K(0,\cdot,\cdot)=W$ provided that the initial value $u^{N}(0,\cdot)$ converges to $u_{0}=u(0,\cdot)$ with respect to $\norm{\cdot}_{L^2}$, i.e.\@ $\lim_{N\to\infty}\norm{u^{N}(0,\cdot)-u(0,\cdot)}_{L^{2}(I)}=0$. More precisely, we have
\begin{equation*}
 \lim_{N\to\infty}\sup_{t\in[0,T]}\Bigl(\norm{u^{N}-u}_{L^{2}(I)}^{2}+\norm{K^{N}(t,\cdot,\cdot)-K(t,\cdot,\cdot)}_{L^{2}}^{2}\Bigr)=0.
\end{equation*}
\end{theorem}

\subsection{Relation to previous results}

Theorem~\ref{Thm:cont:limit} provides the continuum limit for a rather general class of adaptively coupled network dynamics. In particular, it contains as a special case the following system modelling opinion dynamics with time varying weights which has been considered in \cite{AyP21}:
\begin{equation}\label{eq:opinion:dynamics}
 \begin{aligned}
  \dot{\phi}_{k}&=\frac{1}{N}\sum_{\ell=1}^{N}m_{\ell}(t)\psi(\phi_{\ell}-\phi_{k})\\
  \dot{m}_{k}&=\Psi_{k}(\phi,m)
 \end{aligned}
 \qquad\qquad k=1,\ldots,N.
 \end{equation}
Here the opinions are described by $\phi=(\phi_k)_{k=1}^{N}\colon [0,T]\to (\R^{d})^{N}$ while the weights are given by $m=(m_{k})_{k=1}^{N}\colon [0,T]\to \R^{N}$. In fact, for $\kappa_{k\ell}=m_\ell$ for all $k=1,\ldots,N$, $g(t,\phi_k,\phi_\ell)=\psi(\phi_\ell-\phi_k)$ and $\Lambda_{k\ell}(t,\kappa,\phi)=\Psi_{k}(\phi,\kappa_{1\cdot})$ this model is a special case of \eqref{eq:gen:model}. Moreover, Theorem~\ref{Thm:cont:limit} generalises the class of graph dynamics considered in \cite{GKX21}.

\subsection{Outline}

The remainder of the article is structured as follows. In the next section, we will provide the well-posedness of the continuous system \eqref{eq:gen:cont:limit}. The proof relies essentially on an application of the contraction mapping theorem but due to the relatively weak Lipschitz continuity of $f$ and $\Lambda$ some special care is needed. The proof of well-posedness for \eqref{eq:gen:model} follows in the same way and will thus be omitted. In Section~\ref{Sec:continuum:limit} we will then provide the proof of Theorem~\ref{Thm:cont:limit}.

\section{Well-posedness}\label{Sec:well:posed}

We have the following result on the well-posedness of the discrete system~\eqref{eq:gen:model}.

\begin{proposition}\label{Prop:discrete:well:posed}
 Let $N\in\N$, $T>0$ and let $g\colon [0,T]\times (\R^d)^2\to \R^d$ satisfy \eqref{eq:Ass:f:g}. Assume that $f_{k}\colon [0,T]\times (\R^{d})^{N}\to \R^{d}$ satisfies
 \begin{equation*}
   \begin{aligned}
  \abs{f_{k}(t,\phi)}&\leq B_{f}(1+\abs{\phi}) &&\text{and} &  \abs{f_k(t,\phi)-f_k(t,\psi)}&\leq L_{f}\abs{\phi-\psi}
 \end{aligned}
 \end{equation*}
for all $\phi,\psi\in (\R^{d})^{N}$.
Moreover, assume that $\Lambda_{k\ell}\colon [0,\infty)\times \R^{N\times N} \times (\R^d)^{N}\to \R$ is uniformly Lipschitz continuous with respect to the second and third component, i.e.\@ $\abs{\Lambda_{k\ell}(t,\kappa,\phi)-\Lambda_{k\ell}(t,\lambda,\psi)}\leq L_{\Lambda}(\abs{\kappa-\lambda}+\abs{\phi-\psi})$ and satisfies the bound $\abs{\Lambda_{k\ell}(t,\kappa,\phi)}\leq B_{\Lambda}(1+\abs{\kappa})$ for all $k,\ell\in\{1,\ldots,N\}$ uniformly with respect to $t$. Then for each initial condition $(\phi_0,\kappa_0)\in (\R^d)^{N}\times \R^{N\times N}$ the system \eqref{eq:gen:model} has a unique solution $(\phi,\kappa)$ on $[0,T]$.
\end{proposition}

Proposition~\ref{Prop:discrete:well:posed} can be proved similarly as Proposition~\ref{Prop:continuous:well:posed} dealing with the continuous system~\eqref{eq:gen:cont:limit}. We thus omit the proof of Proposition~\ref{Prop:discrete:well:posed}. However, we state the following lemma which guarantees that the functions defined in \eqref{eq:def:Lambda:discrete} satisfy the assumptions in Proposition~\ref{Prop:discrete:well:posed}.

\begin{lemma}\label{Lem:properties:discrete}
 Let $N\in\N$ fixed and let $f$ be as in \eqref{eq:Ass:f:g} and $\Lambda$ as in \eqref{eq:Ass:Lambda}. Then $\Lambda_{k\ell}$ and $f_k$ as defined in \eqref{eq:def:Lambda:discrete} satisfy the assumptions of Proposition~\ref{Prop:discrete:well:posed}.
\end{lemma}

\begin{proof}
 According to \cref{eq:def:Lambda:discrete,eq:Ass:f:g} we have
 \begin{multline*}
  \sum_{k=1}^{N}\abs{f_{k}(t,\phi)}^2=N^2\sum_{k=1}^{N}\abs[\bigg]{\int_{I_{k}}f(t,x,{\textstyle\sum_{\ell=1}^{N}}\phi_{\ell}\chi_{I_{\ell}}(\cdot))\dd{x}}^2\leq B_{f}^2N \Bigl(1+\norm[\Big]{\sum_{\ell=1}^{N}\phi_{\ell}\chi_{I_{\ell}}(\cdot)}_{L^{\infty}}\Bigr)^2\\*
  \leq B_{f}^2N\Bigl(1+\Bigl(\sum_{\ell=1}^{N}\abs{\phi_{\ell}}^2\Bigr)^{1/2}\Bigr)^2=B_{f}^2N\Bigl(1+\abs{\phi}\Bigr)^2.
 \end{multline*}
Similarly, using additionally Cauchy's inequality we find
\begin{multline*}
 \sum_{k=1}^{N}\abs{f_{k}(t,\phi)-f_{k}(t,\psi)}^{2}=N^2\sum_{k=1}^{N}\abs[\bigg]{\int_{I_{k}}f(t,x,{\textstyle\sum_{\ell=1}^{N}}\phi_{\ell}\chi_{I_{\ell}}(\cdot))-f(t,x,{\textstyle\sum_{\ell=1}^{N}}\psi_{\ell}\chi_{I_{\ell}}(\cdot))\dd{x}}^2\\*
 \leq N\sum_{k=1}^{N}\int_{I_{k}}\abs[\big]{f(t,x,{\textstyle\sum_{\ell=1}^{N}}\phi_{\ell}\chi_{I_{\ell}}(\cdot))-f(t,x,{\textstyle\sum_{\ell=1}^{N}}\psi_{\ell}\chi_{I_{\ell}}(\cdot))}^2\dd{x}\\*
 =N\norm[\big]{f(t,\cdot,{\textstyle\sum_{\ell=1}^{N}}\phi_{\ell}\chi_{I_{\ell}}(\cdot))-f(t,\cdot,{\textstyle\sum_{\ell=1}^{N}}\psi_{\ell}\chi_{I_{\ell}}(\cdot))}_{L^{2}(I)}^2\leq NL_{f}^{2}\norm[\big]{{\textstyle\sum_{\ell=1}^{N}}(\phi_{\ell}-\psi_{\ell})\chi_{I_{\ell}}(\cdot))}_{L^{2}(I)}^{2}\\*
 =L_{f}^{2}\abs{\phi-\psi}^{2}.
\end{multline*}
In the same way we get 
\begin{multline*}
 \sum_{k,\ell=1}^{N}\abs{\Lambda_{k\ell}(t,\kappa,\phi)}^{2}=N^4\sum_{k,\ell=1}^{N}\abs[\bigg]{\int_{I_{k}\times I_{\ell}}\Lambda(t,x,y,{\textstyle\sum_{m,n=1}^{N}}\kappa_{mn}\chi_{I_m\times I_{n}}(\cdot),\textstyle{\sum_{m=1}^{N}}\phi_{m}\chi_{I_{m}}(\cdot))}^2\\*
 \leq B_{\Lambda}^2N\Bigl(1+\norm[\Big]{\sum_{m,n=1}^{N}\kappa_{mn}\chi_{I_{m}\times I_{n}}(\cdot)}_{L^{\infty}(I^2)}\Bigr)^{2}\leq B_{\Lambda}^2N\Bigl(1+\abs{\kappa}\Bigr)^{2}.
\end{multline*}
Finally
\begin{multline*}
 \sum_{k,\ell=1}^{N}\abs{\Lambda_{k\ell}(t,\kappa,\phi)-\Lambda_{k\ell}(t,\lambda,\psi)}^{2}\\*
 \shoveleft{=N^4\sum_{k,\ell=1}^{N}\biggl|\int_{I_{k}\times I_{\ell}}\Lambda(t,x,y,{\textstyle\sum_{m,n=1}^{N}}\kappa_{mn}\chi_{I_m\times I_{n}}(\cdot),\textstyle{\sum_{m=1}^{N}}\phi_{m}\chi_{I_{m}}(\cdot))}\\*
 \shoveright{-\Lambda(t,x,y,{\textstyle\sum_{m,n=1}^{N}}\lambda_{mn}\chi_{I_m\times I_{n}}(\cdot),\textstyle{\sum_{m=1}^{N}}\psi_{m}\chi_{I_{m}}(\cdot))\dd{x}\dd{y}\biggr|^2}\\*
 \shoveleft{\leq N^{2}\sum_{k,\ell=1}^{N}\int_{I_{k}\times I_{\ell}}\bigl|\Lambda(t,x,y,{\textstyle\sum_{m,n=1}^{N}}\kappa_{mn}\chi_{I_m\times I_{n}}(\cdot),\textstyle{\sum_{m=1}^{N}}\phi_{m}\chi_{I_{m}}(\cdot))}\\*
 \shoveright{-\Lambda(t,x,y,{\textstyle\sum_{m,n=1}^{N}}\lambda_{mn}\chi_{I_m\times I_{n}}(\cdot),\textstyle{\sum_{m=1}^{N}}\psi_{m}\chi_{I_{m}}(\cdot))\bigr|^2\dd{x}\dd{y}}\\*
 \shoveleft{=N^2 \bigl\|\Lambda(t,x,y,{\textstyle\sum_{m,n=1}^{N}}\kappa_{mn}\chi_{I_m\times I_{n}}(\cdot),\textstyle{\sum_{m=1}^{N}}\phi_{m}\chi_{I_{m}}(\cdot))}\\*
 \shoveright{-\Lambda(t,x,y,{\textstyle\sum_{m,n=1}^{N}}\lambda_{mn}\chi_{I_m\times I_{n}}(\cdot),\textstyle{\sum_{m=1}^{N}}\psi_{m}\chi_{I_{m}}(\cdot))\bigr\|^2_{L^{2}(I^2)}}\\*
 \shoveleft{\leq N^2L_{\Lambda}^{2}\Bigl(\norm[\big]{{\textstyle\sum_{m,n=1}^{N}}(\kappa_{mn}-\lambda_{mn})\chi_{I_{m}\times I_{n}}(\cdot))}_{L^{2}(I^2)}+\norm[\big]{{\textstyle\sum_{m=1}^{N}}(\phi_{m}-\psi_{m})\chi_{I_{m}}(\cdot))}_{L^{2}(I)}\Bigr)^{2}}\\*
 =L_{\Lambda}^{2}\Bigl(\abs{\kappa-\lambda}+N\abs{\phi-\psi}\Bigr)^2.
\end{multline*}
\end{proof}

The following proposition guarantees the existence of a unique solution to the continuum limit equation~\eqref{eq:gen:cont:limit}.

\begin{proposition}\label{Prop:continuous:well:posed}
 Let $T>0$ and assume that $f\colon[0,T]\times I\times L^{\infty}(I,\R^{d})\to \R^{d}$ and $g\colon [0,T]\times (\R^d)^2\to \R^d$ satisfy~\eqref{eq:Ass:f:g}. Moreover, assume that $\Lambda\colon [0,\infty)\times I\times I \times L^{2}(I\times I,\R) \times L^{2}(I,\R^d)\to \R$ satisfies \eqref{eq:Ass:Lambda}. Then for each initial condition $(u_0,K_0)\in L^{\infty}(I,\R^d)\times L^{\infty}(I^2,\R)$ the system \eqref{eq:gen:cont:limit} has a unique solution $(u,K)\in C^1([0,T],L^{\infty}(I,\R^d))\times C^1([0,T],L^{\infty}(I^2,\R))$.
\end{proposition}

 The claim will follow from the contraction mapping theorem. Due to the properties in \eqref{eq:Ass:Lambda}, we can only obtain a contractive operator with respect to $\norm{\cdot}_{L^2}$. However, by following the proof of the contraction mapping theorem and tracking the iterating sequence, we obtain in fact the existence of a unique solution in $L^{\infty}$. A similar argument has been used in \cite{AyP21} relying on a two step procedure, while here, we proceed in one step. For $(u_{t_0},K_{t_0})\in L^{\infty}(I,\R^d)\times L^{\infty}(I^2,\R)$ we define the operator $\mathcal{A}\vcc=(\mathcal{A}_1,\mathcal{A}_2)\colon C([t_0,T],L^{\infty}(I))\times C([t_0,T],L^{\infty}(I^2))\to C([t_0,T],L^{\infty}(I))\times C([t_0,T],L^{\infty}(I^2))$ related to the system~\eqref{eq:gen:cont:limit} via:
 \begin{equation}\label{eq:operator:gen}
  \begin{aligned}
   \mathcal{A}_{1}[u,K](t,x)&\vcc=u_{t_0}(x)+\int_{t_0}^{t}\int_{I}K(s,x,y)g(s,u(s,x),u(s,y))\dd{y}\dd{s}+\int_{t_0}^{t}f(s,x,u(s,\cdot))\dd{s}\\
   \mathcal{A}_{2}[u,K](t,x,y)&\vcc=K_{t_0}(x,y)+\int_{t_0}^{t}\Lambda(s,x,y,K(s,\cdot,\cdot),u(s,\cdot))\dd{s}.
  \end{aligned}
 \end{equation}
 \begin{lemma}\label{Lem:operator:A:bounded}
  The operator 
  \begin{equation*}
   \mathcal{A}\colon C([t_0,T],L^{\infty}(I))\times C([t_0,T],L^{\infty}(I^2))\longrightarrow C([t_0,T],L^{\infty}(I))\times C([t_0,T],L^{\infty}(I^2))
  \end{equation*}
is well-defined.
 \end{lemma}

 \begin{proof}
 By definition $\mathcal{A}[u,K]$ is continuous in time. Thus, to show that $\mathcal{A}$ is well-defined it suffices to show the boundedness. For $(u,K)\in C([t_0,T],L^{\infty}(I))\times C([t_0,T],L^{\infty}(I^2))$, we can estimate $\mathcal{A}_{1}$ as
\begin{multline}\label{eq:A:bounded:1}
 \norm{\mathcal{A}_{1}[u,K](t,\cdot)}_{L^{\infty}(I)}\leq \norm{u_{t_0}}_{L^{\infty}(I)}+B_g\int_{t_0}^{t}\norm{K(s,\cdot,\cdot)}_{L^{\infty}(I^2)}\dd{s}+B_{f}\int_{t_0}^{t}(1+\norm{u(s,\cdot)}_{L^{\infty}(I)})\dd{s}\\*
 \leq  \norm{u_{t_0}}_{L^{\infty}(I)}+\Bigl(B_g\norm{K}_{C([t_0,T],L^{\infty})}+B_{f}(1+\norm{u}_{
 C([t_0,T],L^{\infty})})\Bigr)(t-t_0).
\end{multline}
Thus,
\begin{equation*}
 \norm{\mathcal{A}_{1}[u,K]}_{C([t_0,T],L^{\infty}(I))}\leq \norm{u_{t_0}}_{L^{\infty}(I)}+\Bigl(B_g\norm{K}_{C([t_0,T],L^{\infty})}+B_{f}(1+\norm{u}_{C([t_0,T],L^{\infty})})\Bigr)(T-t_0).
\end{equation*}
Moreover, for $\mathcal{A}_{2}$ we have
\begin{multline*}
 \norm{\mathcal{A}_{2}[u,K](t,\cdot)}_{L^{\infty}(I^2)}\leq \norm{K_{t_0}}_{L^{\infty}(I^2)}+B_{\Lambda}\int_{t_0}^{t}(1+\norm{K(s,\cdot,\cdot)}_{L^{\infty}})\dd{s}\\*
 \leq \norm{K_{t_0}}_{L^{\infty}(I^2)}+B_{\Lambda}\Bigl(1+\norm{K}_{C([t_0,T],L^{\infty})}\Bigr)(t-t_0).
\end{multline*}
Thus,
\begin{equation*}
 \norm{\mathcal{A}_{2}[u,K]}_{C([t_0,T],L^{\infty}(I^2))}\leq \norm{K_{t_0}}_{L^{\infty}(I^2)}+B_{\Lambda}\Bigl(1+\norm{K}_{C([t_0,T],L^{\infty})}\Bigr)(T-t_0).
\end{equation*}
\end{proof}
Lemma~\ref{Lem:operator:A:bounded} allows to define the sequence $(v_n,J_n)_{n\in\N}\subset C([t_0,T],L^{\infty}(I))\times C([t_0,T],L^{\infty}(I^2))$ via
\begin{equation}\label{eq:Iterate}
  (v_{n},J_{n})\vcc=\mathcal{A}^{n}[u_{t_0},K_{t_0}]
\end{equation}
where $\mathcal{A}^{n}$ denotes the $n$-th iterate of the operator $\mathcal{A}$.
We have the following uniform bounds on $(v_n,J_n)_{n\in\N}$.
\begin{lemma}\label{Lem:A:iterate}
 Let $\Lambda$ satisfy \eqref{eq:Ass:Lambda} and let $u_{t_0}\in L^{\infty}(I,\R^d)$ and $K_{t_0}\in L^{\infty}(I^2)$ such that $1+\norm{K_{t_0}}_{L^{\infty}}\leq (1+\norm{K_{0}}_{L^{\infty}})\ee^{B_{\Lambda}t_{0}}$. Then the sequence $(v_n,J_n)_{n\in\N}$ defined in \eqref{eq:Iterate} satisfies 
 \begin{equation*}
  \begin{aligned}
     \norm{J_n(t,\cdot,\cdot)}_{L^{\infty}(I^2)}&\leq (1+\norm{K_{t_0}}_{L^{\infty}})\ee^{B_{\Lambda} (t-t_{0})}-1\leq (1+\norm{K_{0}}_{L^{\infty}})\ee^{B_{\Lambda}t}-1\\
     \norm{v_{n}(t,\cdot)}_{L^{\infty}(I)}&\leq (1+\norm{u_{t_0}}_{L^{\infty}})\ee^{B_{f}(t-t_{0})}+\frac{B_{g}}{B_{\Lambda}-B_{f}}(1+\norm{K_{t_{0}}}_{L^{\infty}})\bigl(\ee^{B_{\Lambda}(t-t_{0})}-\ee^{B_{f}(t-t_{0})}\bigr)-1
  \end{aligned}
 \end{equation*}
 for all $n\in\N_{0}$. In particular we have
  \begin{equation*}
  \begin{aligned}
     \norm{J_n}_{C([t_0,T],L^{\infty}(I^2))}&\leq (1+\norm{K_{0}}_{L^{\infty}})\ee^{B_{\Lambda} T}\\
     \norm{v_{n}(t,\cdot)}_{C([t_{0},T],L^{\infty}(I))}&\leq (1+\norm{u_{t_0}}_{L^{\infty}})\ee^{B_{f}(T-t_{0})}+\frac{B_{g}(1+\norm{K_{t_{0}}}_{L^{\infty}})}{B_{\Lambda}-B_{f}}\bigl(\ee^{B_{\Lambda}(T-t_{0})}-\ee^{B_{f}(T-t_{0})}\bigr)-1.
  \end{aligned}
 \end{equation*}
 \end{lemma}
 
\begin{remark}
 Note that the estimate on $v_n$ makes sense and is also valid in the limiting case $B_{\Lambda}-B_{f}=0$ when it reduces to \begin{equation*}                                                                                                       
  \norm{v_{n}(t,\cdot)}_{L^{\infty}(I)}\leq (1+\norm{u_{t_0}}_{L^{\infty}})\ee^{B_{f}(t-t_{0})}+B_{g}(1+\norm{K_{t_{0}}}_{L^{\infty}})(t-t_{0})-1.
 \end{equation*}
\end{remark}

\begin{proof}[Proof of Lemma~\ref{Lem:A:iterate}]
 The bound on $J_n$ is a direct consequence of the following estimate which we obtain by induction: 
 \begin{equation}\label{eq:est:A2:iterate}
  1+\norm{(\mathcal{A}^{n}[u_{t_0},K_{t_0}])_{2}(t,\cdot,\cdot)}_{L^{\infty}}\leq \Bigl(\sum_{\ell=0}^{n}\frac{B_{\Lambda}^{\ell}}{\ell!}(t-t_0)^{\ell}\Bigr)\bigl(1+\norm{K_{t_0}}_{L^{\infty}}\bigr).
 \end{equation}
Similarly, it follows by induction that
\begin{multline}\label{eq:est:A1:iterate}
 1+\norm{(\mathcal{A}^{n}[u_{t_0},K_{t_0}])_{1}(t,\cdot)}_{L^{\infty}}\leq (1+\norm{u_{t_0}}_{L^{\infty}})\sum_{\ell=0}^{n}\frac{B_{f}^{\ell}}{\ell!}(t-t_{0})^{\ell} \\*
 +\frac{B_{g}}{B_{\Lambda}}(1+\norm{K_{t_{0}}}_{L^{\infty}})\sum_{k=0}^{n-1}\Bigl(\frac{B_{f}}{B_{\Lambda}}\Bigr)^{k}\sum_{\ell=1+k}^{n}\frac{B_{\Lambda}^{\ell}}{\ell!}(t-t_{0})^{\ell}-\frac{B_{g}}{B_{f}}\sum_{\ell=1}^{n}\frac{B_{f}^{\ell}}{\ell!}(t-t_{0})^{\ell}
\end{multline}
with $\sum_{k=0}^{-1}(\cdots)\vcc=0=\vcc \sum_{\ell=1}^{0}(\cdots)$. Moreover, we note that
\begin{equation*}
 \sum_{k=0}^{n-1}\Bigl(\frac{B_{f}}{B_{\Lambda}}\Bigr)^{k}\sum_{\ell=1+k}^{n}\frac{B_{\Lambda}^{\ell}}{\ell!}(t-t_{0})^{\ell}=\sum_{\ell=1}^{n}\sum_{k=0}^{\ell-1}\Bigl(\frac{B_{f}}{B_{\Lambda}}\Bigr)^{k}\frac{B_{\Lambda}^{\ell}}{\ell!}(t-t_{0})^{\ell}=\frac{B_{\Lambda}}{B_{\Lambda}-B_{f}}\sum_{\ell=1}^{n}\frac{B_{\Lambda}^{\ell}-B_{f}^{\ell}}{\ell!}(t-t_{0})^{\ell}.
\end{equation*}
Together with \eqref{eq:est:A1:iterate} we thus get
\begin{multline*}
 1+\norm{v_{n}(t,\cdot)}_{L^{\infty}}\\*
 \leq (1+\norm{u_{t_0}}_{L^{\infty}})\sum_{\ell=0}^{n}\frac{B_{f}^{\ell}}{\ell!}(t-t_{0})^{\ell}+\frac{B_{g}}{B_{\Lambda}-B_{f}}(1+\norm{K_{t_{0}}}_{L^{\infty}})\sum_{\ell=1}^{n}\frac{B_{\Lambda}^{\ell}-B_{f}^{\ell}}{\ell!}(t-t_{0})^{\ell}-\frac{B_{g}}{B_{f}}\sum_{\ell=1}^{n}\frac{B_{f}^{\ell}}{\ell!}(t-t_{0})^{\ell}
\end{multline*}
which finishes the proof
\end{proof}
The next lemma shows that the operator is contractive with respect to the $L^2$ norm.
\begin{lemma}\label{Lem:A:contractive:L2}
 Let $K_{0}\in L^{\infty}(I^2)$ and $t_{0}\in [0,T)$ and assume \cref{eq:Ass:Lambda,eq:Ass:f:g}. For 
 \begin{equation*}
  0<T_{*}\leq \frac{1}{2(2^{5/2}L_{g}(1+\norm{K_{0}}_{L^{\infty}(I^2)})\ee^{B_{\Lambda} T}+L_{f}+\sqrt{2}B_{g}+L_{\Lambda})}
 \end{equation*}
 the operator $\mathcal{A}$ is contractive on the set $\mathcal{S}_{K_0}\vcc=\{(u,K)\in C([t_0,t_{0}+T_{*}],L^{\infty}(I)\times L^{\infty}(I^2)) \mid 1+\norm{K(t,\cdot,\cdot)}_{C([t_0,t_{0}+T_{*}],L^{\infty}(I^2))}\leq (1+\norm{K_0}_{L^{\infty}(I^2)})\ee^{B_{\Lambda}T}\}$  with respect to $\norm{\cdot}_{C([t_{0},t_{0}+T_{*}],L^{2}(I)\times L^{2}(I^2))}$ for each $t_{0}<T$ as long as $t_{0}+T_{*}\leq T$. More precisely, under these conditions we have
 \begin{multline*}
  \norm{\mathcal{A}[u_1,K_1]-\mathcal{A}[u_2,K_2]}_{C([t_{0},t_{0}+T_{*}],L^{2}(I)\times L^{2}(I^2))}\\*
  \leq \frac{1}{2}\Bigl(\norm{u_1-u_2}_{C([t_{0},t_{0}+T_{*}],L^{2}(I))}+\norm{K_1-K_2}_{C([t_{0},t_{0}+T_{*}],L^{2}(I^2))}\Bigr).
 \end{multline*}
\end{lemma}
\begin{proof}
 Let $(u_1,K_1),(u_2,K_2)\in \mathcal{S}_{K_{0}}$. For $\mathcal{A}_{2}$ we get together with Cauchy's inequality and Fubini's Theorem that
 \begin{multline*}
  \norm{\mathcal{A}_{2}[u_{1},K_1](t,\cdot,\cdot)-\mathcal{A}_{2}[u_{2},K_2](t,\cdot,\cdot)}_{L^{2}(I^2)}\\*
  =\biggl(\int_{I^2}\biggl(\int_{t_0}^{t}\Lambda(s,x,y,K_1(s,\cdot,\cdot),u_1(s,\cdot))-\Lambda(s,x,y,K_2(s,\cdot,\cdot),u_2(s,\cdot))\dd{s}\biggr)^{2}\dd{x}\dd{y}\biggr)^{1/2}\\*
  \leq \biggl(\int_{I^2}(t-t_0)\int_{t_0}^{t}\abs[\big]{\Lambda(s,x,y,K_1(s,\cdot,\cdot),u_1(s,\cdot))-\Lambda(s,x,y,K_2(s,\cdot,\cdot),u_2(s,\cdot))}^2\dd{s}\dd{x}\dd{y}\biggr)^{1/2}\\*
  =(t-t_0)^{1/2}\biggl(\int_{t_0}^{t}\norm[\big]{\Lambda(s,x,y,K_1(s,\cdot,\cdot),u_1(s,\cdot))- \Lambda(s,x,y,K_2(s,\cdot,\cdot),u_2(s,\cdot))}^2_{L^{2}(I^2)}\dd{s}\biggr)^{1/2}.
 \end{multline*}
 By means of \eqref{eq:Ass:Lambda} we deduce
 \begin{multline*}
  \norm{\mathcal{A}_{2}[u_1,K_1](t,\cdot,\cdot)-\mathcal{A}_{2}[u_{2},K_2](t,\cdot,\cdot)}_{L^{2}(I^2)}\\*
  \leq L_{\Lambda}(t-t_0)^{1/2}\biggl(\int_{t_0}^{t}\Bigl(\norm[\big]{K_1(s,\cdot,\cdot)-K_2(s,\cdot,\cdot)}_{L^{2}(I^2)}+\norm[\big]{u_1(s,\cdot)-u_2(s,\cdot)}_{L^{2}(I)}\Bigr)^2\dd{s}\biggr)^{1/2}.
 \end{multline*}
This finally yields
\begin{multline}\label{eq:contractivity:A2}
  \norm{\mathcal{A}_{2}[u_1,K_1]-\mathcal{A}_{2}[u_2,K_2]}_{C([t_0,t_0+T_*],L^{2}(I^2))}\\*
  \leq L_{\Lambda}T_{*}\Bigl(\norm[\big]{K_1(s,\cdot,\cdot)-K_2(s,\cdot,\cdot)}_{C([t_0,t_0+T_*],L^{2}(I^2))}+\norm[\big]{u_1(s,\cdot)-u_2(s,\cdot)}_{C([t_0,t_0+T_*],L^{2}(I))}\Bigr).
\end{multline}
For $\mathcal{A}_{1}$ we find similarly by means of Cauchy's inequality and Fubini's Theorem together with \eqref{eq:Ass:f:g} that
\begin{multline*}
 \norm{\mathcal{A}_{1}[u_1,K_1](t,\cdot)-\mathcal{A}_{1}[u_2,K_2](t,\cdot)}_{L^{2}(I)}\\*
 \leq \biggl(\int_{I}\abs[\bigg]{\int_{t_{0}}^{t}\int_{I}K_{1}(s,x,y)g(s,u_1(s,x),u_1(s,y))-K_2(s,x,y)g(s,u_2(s,x),u_2(s,y))\dd{y}\dd{s}}^{2}\dd{x}\biggr)^{1/2}\\*
 \shoveright{+\biggl(\int_{I}\abs[\bigg]{\int_{t_{0}}^{t}f(s,x,u_{1}(s,\cdot))-f(s,x,u_{2}(s,\cdot))\dd{s}}^{2}\dd{x}\biggr)^{1/2}}\\*
 \shoveleft{\leq (t-t_{0})^{1/2}\biggl(\int_{I}\int_{t_{0}}^{t}\int_{I}\Bigl(\norm{K_{1}}_{C([t_{0},t_{0}+T_{*}],L^{\infty})}\abs[\big]{g(s,u_1(s,x),u_1(s,y))-g(s,u_2(s,x),u_2(s,y))}}\\*
 \shoveright{+B_{g}\abs{K_{1}(s,x,y)-K_2(s,x,y)}\Bigr)^{2}\dd{y}\dd{s}\dd{x}\biggr)^{1/2}}\\*
 +L_{f}(t-t_{0})^{1/2}\biggl(\int_{t_{0}}^{t}\norm{u_{1}(s,\cdot)-u_{2}(s,\cdot)}_{L^{2}(I)}^{2}\dd{s}\biggr)^{1/2}.
\end{multline*}
Using \eqref{eq:Ass:f:g} together with Young's inequality and the properties of $\mathcal{S}_{K_0}$ we further deduce
\begin{multline*}
 \norm{\mathcal{A}_{1}[u_1,K_1](t,\cdot)-\mathcal{A}_{1}[u_2,K_2](t,\cdot)}_{L^{2}(I)}\\*
 \shoveleft{\leq \sqrt{2}(t-t_{0})^{1/2}\Biggl[L_{g}(1+\norm{K_{0}}_{L^{\infty}(I^2)})\ee^{B_{\Lambda} T}\biggl(\int_{I}\int_{t_{0}}^{t}\int_{I}\Bigl(\abs{u_1(s,y)-u_2(s,y)}+\abs{u_1(s,x)-u_2(s,x)}\Bigr)^2\biggr)^{1/2}}\\*
 \shoveright{+B_{g}\biggl(\int_{I}\int_{t_{0}}^{t}\int_{I}\abs{K_{1}(s,x,y)-K_2(s,x,y)}^{2}\dd{y}\dd{s}\dd{x}\biggr)^{1/2}\Biggr]}\\*
 +L_{f}(t-t_{0})^{1/2}\biggl(\int_{t_{0}}^{t}\norm{u_{1}(s,\cdot)-u_{2}(s,\cdot)}_{L^{2}(I)}^{2}\dd{s}\biggr)^{1/2}.
\end{multline*}
Cauchy's inequality together with Fubini's Theorem then implies
\begin{multline*}
 \norm{\mathcal{A}_{1}[u_1,K_1](t,\cdot)-\mathcal{A}_{1}[u_2,K_2](t,\cdot)}_{L^{2}(I)}\\*
 \shoveleft{\leq \sqrt{2}(t-t_{0})^{1/2}\Biggl[4L_{g}(1+\norm{K_{0}}_{L^{\infty}(I^2)})\ee^{B_{\Lambda} T}\biggl(\int_{t_{0}}^{t}\norm{u_1(s,\cdot)-u_2(s,\cdot)}_{L^{2}(I)}^{2}\dd{s}\biggr)^{1/2}}\\*
 \shoveright{+B_{g}\biggl(\int_{t_{0}}^{t}\norm{K_{1}(s,\cdot,\cdot)-K_2(s,\cdot,\cdot)}_{L^{2}(I^2)}^{2}\dd{s}\biggr)^{1/2}\Biggr]}\\*
 +L_{f}(t-t_{0})^{1/2}\biggl(\int_{t_{0}}^{t}\norm{u_{1}(s,\cdot)-u_{2}(s,\cdot)}_{L^{2}(I)}^{2}\dd{s}\biggr)^{1/2}.
\end{multline*}
This yields
\begin{multline*}
  \norm{\mathcal{A}_{1}[u_1,K_1]-\mathcal{A}_{1}[u_2,K_2]}_{C([t_{0},t_{0}+T_{*}],L^{2}(I))}\\*
 \leq T_{*}\Bigl(2^{5/2}L_{g}(1+\norm{K_{0}}_{L^{\infty}(I^2)})\ee^{B_{\Lambda} T}+L_{f}\Bigr)\norm{u_1-u_2}_{C([t_{0},t_{0}+T_{*}],L^{2}(I))}\\*
 +\sqrt{2}T_{*}B_{g}\norm{K_{1}-K_2}_{C([t_{0},t_{0}+T_{*}],L^{2}(I^2))}.
\end{multline*}
Together with \eqref{eq:contractivity:A2} we deduce
\begin{multline*}
 \norm{\mathcal{A}[u_1,K_1]-\mathcal{A}[u_2,K_2]}_{C([t_{0},t_{0}+T_{*}],L^{2}(I)\times L^{2}(I^2))}\\*
  \shoveleft{\vcc=\norm{\mathcal{A}_{1}[u_1,K_1]-\mathcal{A}_{1}[u_2,K_2]}_{C([t_{0},t_{0}+T_{*}],L^{2}(I))}+\norm{\mathcal{A}_{2}[u_1,K_1]-\mathcal{A}_{2}[u_2,K_2]}_{C([t_0,t_0+T_*],L^{2}(I^2))}}\\*
  \shoveleft{\leq T_{*}\bigl(2^{5/2}L_{g}(1+\norm{K_{0}}_{L^{\infty}(I^2)})\ee^{B_{\Lambda} T}+L_{f}+\sqrt{2}B_{g}+L_{\Lambda}\bigr)\times}\\*
  \times\Bigl(\norm{u_1-u_2}_{C([t_{0},t_{0}+T_{*}],L^{2}(I))}+\norm{K_1-K_2}_{C([t_{0},t_{0}+T_{*}],L^{2}(I^2))}\Bigr).
\end{multline*}
Thus, for 
\begin{equation*}
 T_{*}\leq \frac{1}{2(2^{5/2}L_{g}(1+\norm{K_{0}}_{L^{\infty}(I^2)})\ee^{B_{\Lambda} T}+L_{f}+\sqrt{2}B_{g}+L_{\Lambda})}
\end{equation*}
the claim follows.
\end{proof}
Moreover, we have the following a-priori estimate on solutions of the system \eqref{eq:gen:cont:limit}.

\begin{lemma}\label{Lem:K:apriori}
 Assume that $f\colon [0,T]\times I \times L^{\infty}(I,\R^{d})\to \R^{d}$ and $g\colon[0,T]\times (\R^{d})^2\to \R^{d}$ satisfy~\eqref{eq:Ass:f:g}. Moreover, assume that $\Lambda\colon [0,\infty)\times I\times I \times L^{\infty}(I\times I,\R) \times L^{\infty}(I,\R^d)\to \R$ satisfies \eqref{eq:Ass:Lambda}. Let $(u_{t_0},K_{t_0})\in L^{\infty}(I,\R^d)\times L^{\infty}(I^2,\R)$. Let $(u,K)$ solve \eqref{eq:gen:cont:limit} on $[t_{0},T_1]$ with $0\leq t_0< T_1\leq T$ and initial condition $(u(t_0,\cdot),K(t_0,\cdot,\cdot))=(u_{t_0},K_{t_0})$. Then, we have the estimates 
  \begin{equation*}
  \begin{aligned}
     \norm{u(t,\cdot)}_{L^{\infty}(I)}&\leq \bigl(1+\norm{u_{t_0}}_{L^{\infty}}\bigr)\ee^{B_f (t-t_0)}+\frac{B_g}{B_\Lambda-B_{f}}\bigl(1+\norm{K_0}_{L^{\infty}}\bigr)\bigl(\ee^{(B_{\Lambda}(t-t_0)}-\ee^{B_f(t-t_0)}\bigr)-1\\
    \norm{K(t,\cdot,\cdot)}_{L^{\infty}(I^2)}&\leq (1+\norm{K_{t_0}}_{L^{\infty}(I^2)})\ee^{B_{\Lambda}(t-t_{0})}-1.
  \end{aligned}
 \end{equation*}
 In particular, we have the  bounds
 \begin{equation*}
  \begin{aligned}
     \sup_{t\in [t_0,T_1]}\norm{u(t,\cdot)}_{L^{\infty}(I)}&\leq \bigl(1+\norm{u_{t_0}}_{L^{\infty}}\bigr)\ee^{B_f T_1}+\frac{B_g}{B_\Lambda-B_{f}}\bigl(1+\norm{K_0}_{L^{\infty}}\bigr)\bigl(\ee^{(B_{\Lambda}T_1}-\ee^{B_f T_1}\bigr)-1\\
     \sup_{t\in [t_0,T_1]}\norm{K(t,\cdot,\cdot)}_{L^{\infty}(I^2)}&\leq (1+\norm{K_{t_0}}_{L^{\infty}(I^2)})\ee^{B_{\Lambda}T_1}-1.
  \end{aligned}
 \end{equation*}

\end{lemma}

\begin{proof}
 We start with the estimate on $K$. Since $(u,K)$ solves \eqref{eq:gen:cont:limit}, we have
 \begin{equation*}
  K(t,x,y)=K_{t_0}(x,y)+\int_{0}^{t}\Lambda(s,x,y,K(s,\cdot,\cdot),u(s,\cdot))\dd{s}.
 \end{equation*}
 By means of~\eqref{eq:Ass:Lambda} we get 
 \begin{equation*}
  \norm{K(t,\cdot,\cdot)}_{L^{\infty}(I^2)}\leq \norm{K_{t_0}}_{L^{\infty}(I^2)}+B_{\Lambda}\int_{t_0}^{t}(1+\norm{K(s,\cdot,\cdot)}_{L^{\infty}(I^2)})\dd{s}.
 \end{equation*}
 Gronwall's inequality then implies
 \begin{equation}\label{eq:est:K}
  1+\norm{K(t,\cdot,\cdot)}_{L^{\infty}(I^2)}\leq (1+\norm{K_{t_0}}_{L^{\infty}(I^2)})\ee^{B_{\Lambda}(t-t_{0})}.
 \end{equation}
 With this, the estimate on $u$ follows similarly noting first that
 \begin{equation*}
  u(t,x)=u_{t_0}+\int_{t_0}^{t}\int_{I}K(s,x,y)g(t,u(s,x),u(s,y))\dd{y}\dd{s}+\int_{t_0}^{t}f(s,x,u(s,\cdot))\dd{s}.
 \end{equation*}
 Thus, using again \eqref{eq:Ass:Lambda}, we get together with \eqref{eq:est:K} that
 \begin{equation*}
  \norm{u(t,\cdot)}_{L^{\infty}}\leq \norm{u_{t_0}}_{L^{\infty}}+B_g\int_{t_0}^{t}\Bigl((1+\norm{K_{t_0}}_{L^{\infty}(I^2)})\ee^{B_{\Lambda}(s-t_{0})}-1\Bigr)\dd{s}+B_{f}\int_{t_0}^{t}(1+\norm{u(s,\cdot)}_{L^{\infty}})\dd{s}.
 \end{equation*}
 By means of Gronwall's inequality one deduces that
 \begin{multline}
  1+\norm{u(t,\cdot)}_{L^{\infty}}\\*
  \leq \bigl(1+\norm{u_{t_0}}_{L^{\infty}}\bigr)\ee^{B_f (t-t_0)}+\frac{B_g}{B_\Lambda-B_{f}}\bigl(1+\norm{K_{t_0}}_{L^{\infty}}\bigr)\bigl(\ee^{(B_{\Lambda}(t-t_0)}-\ee^{B_f(t-t_0)}\bigr)-\frac{B_{g}}{B_{f}}(\ee^{B_{f}(t-t_0)}-1)
 \end{multline}
from which the claim follows.
\end{proof}
We can now give the proof of Proposition~\ref{Prop:continuous:well:posed}.

\begin{proof}[Proof of Proposition~\ref{Prop:continuous:well:posed}]
 As announced earlier, we argue along the lines of the proof of the classical contraction mapping theorem. However, since the operator $\mathcal{A}$ is only contractive with respect to the $L^2$ topology, some adjustments are needed. First, we fix $T_{*}\leq T$ according to Lemma~\ref{Lem:A:contractive:L2}. Next, we set $t_0=0$ and define the corresponding sequence $(v_n,J_n)$ as in \eqref{eq:Iterate} which is well-defined according to Lemma~\ref{Lem:operator:A:bounded}. Moreover, due to Lemma~\ref{Lem:A:iterate} the sequence $(J_n)_{n\in\N}$ is uniformly bounded in $C([0,T_*],L^{\infty}(I^2))$ with
 \begin{equation*}
  1+\norm{J_n(t,\cdot,\cdot)}_{C([0,T_*],L^{\infty}(I^2))}\leq (1+\norm{K_{0}}_{L^{\infty}})\ee^{B_{\Lambda} T_{*}}\leq (1+\norm{K_{0}}_{L^{\infty}})\ee^{B_{\Lambda} T} \qquad \text{for all }n\in\N_{0}.
 \end{equation*}
 Consequently, $(v_n,J_n)\in \mathcal{S}_{K_{0}}$ for all $n\in\N_{0}$ with $\mathcal{S}_{K_{0}}$ defined in Lemma~\ref{Lem:A:contractive:L2}. Thus according to this result, we have
 \begin{multline*}
  \norm{v_{n+1}-v_{n}}_{C([0,T_{*}],L^{2}(I))}+\norm{J_{n+1}-J_n}_{C([0,T_{*}],L^{2}(I^2))}\\*
   \leq \frac{1}{2}\Bigl(\norm{v_{n}-v_{n-1}}_{C([0,T_{*}],L^{2}(I))}+\norm{J_{n}-J_{n-1}}_{C([0,T_{*}],L^{2}(I^2))}\Bigr)
 \end{multline*}
which yields by iteration that  $(v_n,J_n)_{n\in \N}$ is a Cauchy sequence in $C([0,T_*],L^{2}(I)\times L^{2}(I^2))$. Consequently, there exists $(u,K)\in C([0,T_*],L^{2}(I)\times L^{2}(I^2))$ such that 
\begin{equation}\label{eq:iterate:limit}
 \norm{v_{n}-u}_{C([0,T_{*}],L^{2}(I))}+\norm{J_n-K}_{C([0,T_{*}],L^{2}(I^2))}\longrightarrow 0 \qquad \text{as } n\to\infty.
\end{equation}
For each $t\in [0,T_*]$ we then have
\begin{equation*}
 v_n(t,\cdot)\to u(t,\cdot) \quad \text{and}\quad J_{n}(t,\cdot,\cdot)\to K(t,\cdot,\cdot) \qquad \text{pointwise almost everywhere as } n\to\infty.
\end{equation*}
Thus, by means of Lemma~\ref{Lem:A:iterate} we have
\begin{equation}\label{eq:bounds:limit}
 \begin{aligned}
     \norm{u(t,\cdot)}_{L^{\infty}(I)}&\leq (1+\norm{u_0}_{L^{\infty}})\ee^{B_{f}t}+\frac{B_{g}}{B_{\Lambda}-B_{f}}(1+\norm{K_{0}}_{L^{\infty}})\bigl(\ee^{B_{\Lambda}t}-\ee^{B_{f}t)}\bigr)-1\\
     \norm{K(t,\cdot,\cdot)}_{L^{\infty}(I^2)}&\leq (1+\norm{K_{0}}_{L^{\infty}})\ee^{B_{\Lambda} t}-1.
 \end{aligned}
\end{equation}
Moreover, as a consequence of \eqref{eq:iterate:limit} we have $(u,K)=\mathcal{A}[u,K]$ and the structure of $\mathcal{A}$ thus immediately implies $(u,K)\in C^1([0,T_*],L^{\infty}(I)\times L^{\infty}(I^2))$ and $(u,K)$ is a solution of \eqref{eq:gen:cont:limit} on $[0,T_{*}]$. Due to \eqref{eq:bounds:limit} we have in particular $(u,K)\in\mathcal{S}_{K_0}$ and according to Lemma~\ref{Lem:K:apriori} any solution $(\hat{u},\hat{K})$ to \eqref{eq:gen:cont:limit} satisfies $(\hat{u},\hat{K})\in\mathcal{S}_{K_0}$. Thus uniqueness follows again from the contractivity in Lemma~\ref{Lem:A:contractive:L2} analogously to the classical contraction mapping theorem. To finish the proof, it remains to extend the solution to $[0,T]$ which can be done, as usual, by iterating the above procedure while we note that Lemma~\ref{Lem:A:iterate} ensures that the condition in the definition of $\mathcal{S}_{K_{0}}$ is preserved.
\end{proof}

\section{The continuum limit}\label{Sec:continuum:limit}

In this section we will give the proof of Theorem~\ref{Thm:cont:limit} using similar arguments as \cite{Med14,AyP21}.

\begin{proof}[Proof of Theorem~\ref{Thm:cont:limit}]
By means of \cref{eq:def:u,eq:gen:cont:limit} we have
 \begin{multline*}
  \frac{1}{2}\del_{t}\norm{u^{N}-u}_{L^{2}(I)}^{2}=\int_{I}\del_{t}\bigl(u^{N}(t,x)-u(t,x)\bigr)\bigl(u^{N}(t,x)-u(t,x)\bigr)\dd{x}\\*
  \shoveleft{=\int_{I^2}\Bigl[K^{N}(t,x,y)g(t,u^{N}(t,x),u^{N}(t,y))-K(t,x,y)g(t,u(t,x),u(t,y))}\Bigr]
  \bigl(u^{N}(t,x)-u(t,x)\bigr)\dd{y}\dd{x}\\*
  +\int_{I}\biggl[N\int_{\lfloor Nx\rfloor/N}^{(\lfloor Nx\rfloor+1)/N}f(t,\xi,u^{N}(t,\cdot))-f(t,x,u(t,\cdot))\dd{\xi}\biggr]\bigl(u^{N}(t,x)-u(t,x)\bigr)\dd{x}.
 \end{multline*}
 Rewriting, we get
  \begin{multline*}
  \frac{1}{2}\del_{t}\norm{u^{N}-u}_{L^{2}(I)}^{2}=\int_{I^2}\Bigl[\bigl(K^{N}(t,x,y)-K(t,x,y)\bigr)g(t,u^{N}(t,x),u^{N}(t,y))\\*
  +K(t,x,y)\bigl(g(t,u^{N}(t,x),u^{N}(t,y))-g(t,u(t,x),u(t,y))\bigr)\Bigr]\bigl(u^{N}(t,x)-u(t,x)\bigr)\dd{y}\dd{x}\\*
  \shoveleft{+\int_{I}\biggl[N\int_{\lfloor Nx\rfloor/N}^{(\lfloor Nx\rfloor+1)/N}f(t,\xi,u^{N}(t,\cdot))-f(t,\xi,u(t,\cdot))}\\*
  +f(t,\xi,u(t,\cdot))-f(t,x,u(t,\cdot))\dd{\xi}\biggr]\bigl(u^{N}(t,x)-u(t,x)\bigr)\dd{x}.
 \end{multline*}
Using the bounds on $f$ and $g$ from \eqref{eq:Ass:f:g} together with Cauchy's inequality we can estimate the right-hand side to get
  \begin{multline*}
  \frac{1}{2}\del_{t}\norm{u^{N}-u}_{L^{2}(I)}^{2}\leq B_{g}\int_{I}\biggl(\int_{I}(K^{N}(t,x,y)-K(t,x,y))^{2}\dd{y}\biggr)^{1/2}\abs{u^{N}(t,x)-u(t,x)}\dd{x}\\*
  +L_{g}\norm{K(t,\cdot,\cdot)}_{L^{\infty}(I^2)}\int_{I^2}\bigl(\abs{u^{N}(y)-u(y)}+\abs{u(x)-u^{N}(x)}\bigr)\abs{u^{N}(t,x)-u(t,x)}\dd{y}\dd{x}\\*
  +\Biggl(\biggl(\int_{I}\abs[\Big]{N\int_{\lfloor Nx\rfloor/N}^{(\lfloor Nx\rfloor+1)/N}f(t,\xi,u^{N}(t,\cdot))-f(t,\xi,u(t,\cdot))\dd{\xi}}^{2}\dd{x}\biggr)^{1/2}\\*
  +\int_{I}\abs[\Big]{N\int_{\lfloor Nx\rfloor/N}^{(\lfloor Nx\rfloor+1)/N}f(t,\xi,u(t,\cdot))-f(t,x,u(t,\cdot))\dd{\xi}}^{2}\dd{x}\biggr)^{1/2}\Biggr)\norm{u^{N}(t,\cdot)-u(t,\cdot)}_{L^{2}(I)}.
 \end{multline*}
 Applying Cauchy's inequality again, we further deduce together with Fubini's Theorem that
   \begin{multline*}
  \frac{1}{2}\del_{t}\norm{u^{N}-u}_{L^{2}(I)}^{2}\leq B_{g}\norm{K^{N}(t,\cdot,\cdot)-K(t,\cdot,\cdot)}_{L^{2}(I^2)}\norm{u^{N}(t,\cdot)-u(t,\cdot)}_{L^{2}(I)}\\*
  +2L_{g}\norm{K(t,\cdot,\cdot)}_{L^{\infty}(I^2)}\norm{u^{N}(t,\cdot)-u(t,\cdot)}_{L^{2}(I)}^{2}\\*
  +\Biggl(\biggl(\int_{I}N\int_{\lfloor N\xi\rfloor/N}^{(\lfloor N\xi\rfloor+1)/N}\abs[\big]{f(t,\xi,u^{N}(t,\cdot))-f(t,\xi,u(t,\cdot))}^{2}\dd{x}\dd{\xi}\biggr)^{1/2}\\*
  +\int_{I}\abs[\Big]{N\int_{\lfloor Nx\rfloor/N}^{(\lfloor Nx\rfloor+1)/N}f(t,\xi,u(t,\cdot))-f(t,x,u(t,\cdot))\dd{\xi}}^{2}\dd{x}\biggr)^{1/2}\Biggr)\norm{u^{N}(t,\cdot)-u(t,\cdot)}_{L^{2}(I)}.
 \end{multline*}
We set 
\begin{equation*}
 r_{N}\vcc=N\int_{\lfloor Nx\rfloor/N}^{(\lfloor Nx\rfloor+1)/N}f(t,\xi,u(t,\cdot))-f(t,x,u(t,\cdot))\dd{\xi}
\end{equation*}
such that Young's inequality together with \eqref{eq:Ass:f:g} then implies
\begin{multline}\label{eq:gen:cont:lim:1}
 \frac{1}{2}\del_{t}\norm{u^{N}-u}_{L^{2}(I)}^{2}\leq \frac{B_{g}}{2}\norm{K^{N}(t,\cdot,\cdot)-K(t,\cdot,\cdot)}_{L^{2}(I^2)}^2\\*
 +\Bigl(2L_{g}\norm{K(t,\cdot,\cdot)}_{L^{\infty}(I^2)}+\frac{B_{g}}{2}\Bigr)\norm{u^{N}(t,\cdot)-u(t,\cdot)}_{L^{2}(I)}^{2}\\*
 +\Bigl(\norm{f(t,\cdot,u^{N}(t,\cdot))-f(t,\cdot,u(t,\cdot))}_{L^{2}(I)}+\norm{r_{N}}_{L^{2}(I)}\Bigr)\norm{u^{N}(t,\cdot)-u(t,\cdot)}_{L^{2}(I)}\\*
 \shoveleft{\leq\frac{B_{g}}{2}\norm{K^{N}(t,\cdot,\cdot)-K(t,\cdot,\cdot)}_{L^{2}(I^2)}^2}\\*
 +\Bigl(2L_{g}\norm{K(t,\cdot,\cdot)}_{L^{\infty}(I^2)}+L_{f}+\frac{B_{g}+1}{2}\Bigr)\norm{u^{N}(t,\cdot)-u(t,\cdot)}_{L^{2}(I)}^{2}+\frac{1}{2}\norm{r_{N}}_{L^{2}(I)}^{2}.
\end{multline}
Similarly, we deduce from \cref{eq:def:u,eq:gen:cont:limit} that
\begin{multline*}
 \frac{1}{2}\del_{t}\norm{K^{N}(t,\cdot,\cdot)-K(t,\cdot,\cdot)}_{L^{2}}^{2}=\int_{I^2}\del_{t}\bigl(K^{N}(t,x,y)-K(t,x,y)\bigr)\bigl(K^{N}(t,x,y)-K(t,x,y)\bigr)\dd{y}\dd{x}\\*
  =\int_{I^2}\biggl[N^2 \int_{\frac{\lfloor Nx\rfloor}{N}}^{\frac{\lfloor Nx\rfloor+1}{N}}\int_{\frac{\lfloor Ny\rfloor}{N}}^{\frac{\lfloor Ny\rfloor+1}{N}}\Lambda(t,\xi,\eta,K^{N}(t,\cdot,\cdot),u^{N}(t,\cdot))-\Lambda(t,x,y,K(t,\cdot,\cdot),u(t,\cdot))\dd{\xi}\dd{\eta}\biggr]\times\\
 \times\bigl(K^{N}(t,x,y)-K(t,x,y)\bigr)\dd{y}\dd{x}. 
\end{multline*}
Together with Cauchy's inequality we can estimate the right-hand side as
\begin{multline}\label{eq:gen:cont:limit:2}
 \frac{1}{2}\del_{t}\norm{K^{N}(t,\cdot,\cdot)-K(t,\cdot,\cdot)}_{L^{2}}^{2}\\*
 \shoveleft{\leq\Biggl[\biggl(\int_{I^2}\biggl(N^2\int_{\frac{\lfloor Nx\rfloor}{N}}^{\frac{\lfloor Nx\rfloor+1}{N}}\int_{\frac{\lfloor Ny\rfloor}{N}}^{\frac{\lfloor Ny\rfloor+1}{N}}\Lambda(t,\xi,\eta,K^{N}(t,\cdot,\cdot),u^{N}(t,\cdot))}\\*\shoveright{-\Lambda(t,\xi,\eta,K(t,\cdot,\cdot),u(t,\cdot))\dd{\xi}\dd{\eta}\biggr)^2\dd{x}\dd{y}\biggr)^{1/2}}\\*
 \shoveleft{+\biggl(\int_{I^2}\biggl(N^2\int_{\frac{\lfloor Nx\rfloor}{N}}^{\frac{\lfloor Nx\rfloor+1}{N}}\int_{\frac{\lfloor Ny\rfloor}{N}}^{\frac{\lfloor Ny\rfloor+1}{N}}\Lambda(t,\xi,\eta,K(t,\cdot,\cdot),u(t,\cdot))\dd{\xi}\dd{\eta}}\\*
 \shoveright{-\Lambda(t,x,y,K(t,\cdot,\cdot),u(t,\cdot))\biggr)^2\dd{x}\dd{y}\biggr)^{1/2}\Biggr]\times}\\*
 \shoveright{\times\norm{(K^{N}(t,\cdot,\cdot)-K(t,\cdot,\cdot)}_{L^2}}\\*
 =\vcc \bigl[\norm{Q_N}_{L^2}+\norm{R_N}_{L^2}\bigr]\norm{(K^{N}(t,\cdot,\cdot)-K(t,\cdot,\cdot)}_{L^2}.
\end{multline}
To estimate the integral given by $Q_N$ further, we apply once more Cauchy's inequality and use Fubini's Theorem to deduce together with \eqref{eq:Ass:Lambda} that
\begin{multline}\label{eq:gen:cont:limit:3}
 \norm{Q_N}_{L^2}\\*
 \shoveleft{\leq \biggl(\int_{I^2}N^2\int_{\lfloor Nx\rfloor/N}^{(\lfloor Nx\rfloor+1)/N}\int_{\lfloor Ny\rfloor/N}^{(\lfloor Ny\rfloor+1)/N}\Bigl(\Lambda(t,\xi,\eta,K^{N}(t,\cdot,\cdot),u^{N}(t,\cdot))}\\*
 \shoveright{-\Lambda(t,\xi,\eta,K(t,\cdot,\cdot),u(t,\cdot))\Bigr)^2\dd{\xi}\dd{\eta}\dd{x}\dd{y}\biggr)^{1/2}}\\*
 \shoveleft{= \biggl(\int_{I^2}N^2\int_{\lfloor N\xi\rfloor/N}^{(\lfloor N\xi\rfloor+1)/N}\int_{\lfloor N\eta\rfloor/N}^{(\lfloor N\eta\rfloor+1)/N}\Bigl(\Lambda(t,\xi,\eta,K^{N}(t,\cdot,\cdot),u^{N}(t,\cdot))}\\*
 \shoveright{-\Lambda(t,\xi,\eta,K(t,\cdot,\cdot),u(t,\cdot))\Bigr)^2\dd{x}\dd{y}\dd{\xi}\dd{\eta}\biggr)^{1/2}}\\*
 \shoveleft{= \biggl(\int_{I^2}\Bigl(\Lambda(t,\xi,\eta,K^{N}(t,\cdot,\cdot),u^{N}(t,\cdot))-\Lambda(t,\xi,\eta,K(t,\cdot,\cdot),u(t,\cdot))\Bigr)^2\dd{\xi}\dd{\eta}\biggr)^{1/2}}\\*
 \shoveleft{= \norm[\big]{\Lambda(t,\cdot,\cdot,K^{N}(t,\cdot,\cdot),u^{N}(t,\cdot))-\Lambda(t,\cdot,\cdot,K(t,\cdot,\cdot),u(t,\cdot))}_{L^2}}\\*
 \leq L_{\Lambda}\bigl(\norm{K^{N}(t,\cdot,\cdot)-K(t,\cdot,\cdot)}_{L^2}+\norm{u^{N}(t,\cdot)-u(t,\cdot)}_{L^2}\bigr).
\end{multline}
Summarising \cref{eq:gen:cont:limit:2,eq:gen:cont:limit:3} we obtain together with Young's inequality that
\begin{multline*}
 \frac{1}{2}\del_{t}\norm{K^{N}(t,\cdot,\cdot)-K(t,\cdot,\cdot)}_{L^{2}}^{2}\\*
 \leq \bigl[L_{\Lambda}\bigl(\norm{K^{N}(t,\cdot,\cdot)-K(t,\cdot,\cdot)}_{L^2}+\norm{u^{N}(t,\cdot)-u(t,\cdot)}_{L^2}\bigr)+\norm{R_N}_{L^2}\bigr]\norm{(K^{N}(t,\cdot,\cdot)-K(t,\cdot,\cdot)}_{L^2}\\*
 \leq \frac{3L_{\Lambda}+1}{2}\norm{(K^{N}(t,\cdot,\cdot)-K(t,\cdot,\cdot)}_{L^2}^2+\frac{L_{\Lambda}}{2}\norm{u^{N}(t,\cdot)-u(t,\cdot)}_{L^2}^2+\frac{1}{2}\norm{R_N}_{L^2}^2.
\end{multline*}
Together with \eqref{eq:gen:cont:lim:1} this yields
\begin{multline*}
 \frac{1}{2}\del_{t}\Bigl(\norm{u^{N}-u}_{L^{2}(I)}^{2}+\norm{K^{N}(t,\cdot,\cdot)-K(t,\cdot,\cdot)}_{L^{2}}^{2}\Bigr)\\*
 \leq \Bigl(2L_{g}\norm{K(t,\cdot,\cdot)}_{L^{\infty}(I^2)}+L_{f}+\frac{B_{g}+L_{\Lambda}+1}{2}\Bigr)\norm{u^{N}(t,\cdot)-u(t,\cdot)}_{L^{2}(I)}^{2}\\*
 +\frac{3L_{\Lambda}+B_{g}+1}{2}\norm{K^{N}(t,\cdot,\cdot)-K(t,\cdot,\cdot)}_{L^{2}(I^2)}^2+\frac{1}{2}\bigl(\norm{r_N}_{L^2}^2+\norm{R_N}_{L^2}^2\bigr)\\*
 \leq \Bigl(2L_{g}\norm{K(t,\cdot,\cdot)}_{L^{\infty}(I^2)}+L_{f}+\frac{3L_{\Lambda}+B_{g}+1}{2}\Bigr)\Bigl(\norm{u^{N}(t,\cdot)-u(t,\cdot)}_{L^{2}(I)}^{2}+\norm{K^{N}(t,\cdot,\cdot)-K(t,\cdot,\cdot)}_{L^{2}}^{2}\Bigr)\\*
 +\frac{1}{2}\bigl(\norm{r_N}_{L^2}^2+\norm{R_N}_{L^2}^2\bigr).
\end{multline*}
Integrating this inequality, we find
\begin{multline*}
 \norm{u^{N}-u}_{L^{2}(I)}^{2}+\norm{K^{N}(t,\cdot,\cdot)-K(t,\cdot,\cdot)}_{L^{2}}^{2}\\*
 \leq \Bigl(\norm{u^{N}(0,\cdot)-u(0,\cdot)}_{L^{2}(I)}^{2}+\norm{K^{N}(0,\cdot,\cdot)-W(\cdot,\cdot)}_{L^{2}}^{2}\Bigr)\ee^{4L_{g}\int_{0}^{t}\norm{K(s,\cdot,\cdot)}_{L^{\infty}(I^2)}\dd{s}+(2L_{f}+3L_{\Lambda}+B_{g}+1)t}\\*
 +\int_{0}^{t}\bigl(\norm{r_N(s)}_{L^2}^2+\norm{R_N(s)}_{L^2}^2\bigr)\ee^{4L_{g}\int_{s}^{t}\norm{K(\tau,\cdot,\cdot)}_{L^{\infty}(I^2)}\dd{\tau}+(2L_{f}+3L_{\Lambda}+B_{g}+1)(t-s)}\dd{s}.
\end{multline*}
On a fixed time interval $[0,T]$ we can estimate the right-hand side uniformly as
\begin{multline}\label{eq:gen:cont:limit:5}
 \norm{u^{N}-u}_{L^{2}(I)}^{2}+\norm{K^{N}(t,\cdot,\cdot)-K(t,\cdot,\cdot)}_{L^{2}}^{2}\\*
 \leq \Bigl(\norm{u^{N}(0,\cdot)-u(0,\cdot)}_{L^{2}(I)}^{2}+\norm{K^{N}(0,\cdot,\cdot)-W(\cdot,\cdot)}_{L^{2}}^{2}+\int_{0}^{T}\bigl(\norm{r_N(s)}_{L^2}^2+\norm{R_N(s)}_{L^2}^2\bigr)\dd{s}\Bigr)\times\\*
 \times \ee^{4L_{g}\int_{0}^{T}\norm{K(s,\cdot,\cdot)}_{L^{\infty}(I^2)}\dd{s}+(2L_{f}+3L_{\Lambda}+B_{g}+1)T}.
\end{multline}
From Lebesgue's differentiation theorem together with dominated convergence we deduce 
\begin{equation*}\label{eq:gen:cont:limit:4}
 \norm{r_{N}}_{L^{2}(I)}\to 0 \quad \text{and}\quad \norm{R_N}_{L^2}\to 0 \qquad \text{as } N\to \infty.
\end{equation*}
Thus, by dominated convergence, for $N\to\infty$ the term in parenthesis in \eqref{eq:gen:cont:limit:5} converges to zero which finishes the proof.
\end{proof}


\begin{thebibliography}{10}

\bibitem{AyP21}
Nathalie Ayi and Nastassia Pouradier~Duteil.
\newblock Mean-field and graph limits for collective dynamics models with
  time-varying weights.
\newblock {\em J. Differential Equations}, 299:65--110, 2021.

\bibitem{BBV08}
Alain Barrat, Marc Barth\'{e}lemy, and Alessandro Vespignani.
\newblock {\em Dynamical processes on complex networks}.
\newblock Cambridge University Press, Cambridge, 2008.

\bibitem{BFK19}
Rico Berner, Jan Fialkowski, Dmitry Kasatkin, Vladimir Nekorkin, Serhiy
  Yanchuk, and Eckehard Schöll.
\newblock Self-similar hierarchical frequency clusters in adaptive networks of
  phase oscillators.
\newblock {\em Preprint arXiv:1904.06927}, 2019.

\bibitem{Bur22}
Martin Burger.
\newblock Kinetic equations for processes on co-evolving networks.
\newblock {\em Kinetic and Related Models}, 15(2):187, 2022.

\bibitem{CFT10}
Jos\'{e}~A. Carrillo, Massimo Fornasier, Giuseppe Toscani, and Francesco Vecil.
\newblock Particle, kinetic, and hydrodynamic models of swarming.
\newblock In {\em Mathematical modeling of collective behavior in
  socio-economic and life sciences}, Model. Simul. Sci. Eng. Technol., pages
  297--336. Birkh\"{a}user Boston, Boston, MA, 2010.

\bibitem{CFL09}
Claudio Castellano, Santo Fortunato, and Vittorio Loreto.
\newblock Statistical physics of social dynamics.
\newblock {\em Rev. Mod. Phys.}, 81:591--646, May 2009.

\bibitem{GKX21}
Marios~Antonios Gkogkas, Christian Kuehn, and Chuang Xu.
\newblock Continuum limits for adaptive network dynamics.
\newblock {\em Preprint arXiv:2109.05898}, 2021.

\bibitem{Gol16}
Fran\c{c}ois Golse.
\newblock On the dynamics of large particle systems in the mean field limit.
\newblock In {\em Macroscopic and large scale phenomena: coarse graining, mean
  field limits and ergodicity}, volume~3 of {\em Lect. Notes Appl. Math.
  Mech.}, pages 1--144. Springer, [Cham], 2016.

\bibitem{HNP16}
Seung-Yeal Ha, Se~Eun Noh, and Jinyeong Park.
\newblock Synchronization of {K}uramoto oscillators with adaptive couplings.
\newblock {\em SIAM J. Appl. Dyn. Syst.}, 15(1):162--194, 2016.

\bibitem{JaK01}
S.~Jain and S.~Krishna.
\newblock A model for the emergence of cooperation, interdependence, and
  structure in evolving networks.
\newblock {\em Proceedings of the National Academy of Sciences},
  98(2):543--547, January 2001.

\bibitem{KaM17}
Dmitry Kaliuzhnyi-Verbovetskyi and Georgi~S. Medvedev.
\newblock The semilinear heat equation on sparse random graphs.
\newblock {\em SIAM J. Math. Anal.}, 49(2):1333--1355, 2017.

\bibitem{KaM18}
Dmitry Kaliuzhnyi-Verbovetskyi and Georgi~S. Medvedev.
\newblock The {M}ean {F}ield {E}quation for the {K}uramoto {M}odel on {G}raph
  {S}equences with {N}on-{L}ipschitz {L}imit.
\newblock {\em SIAM J. Math. Anal.}, 50(3):2441--2465, 2018.

\bibitem{KuT19}
Christian Kuehn and Sebastian Throm.
\newblock Power {N}etwork {D}ynamics on {G}raphons.
\newblock {\em SIAM J. Appl. Math.}, 79(4):1271--1292, 2019.

\bibitem{Kur84}
Y.~Kuramoto.
\newblock {\em Chemical oscillations, waves, and turbulence}, volume~19 of {\em
  Springer Series in Synergetics}.
\newblock Springer-Verlag, Berlin, 1984.

\bibitem{Med14}
Georgi~S. Medvedev.
\newblock The nonlinear heat equation on dense graphs and graph limits.
\newblock {\em SIAM J. Math. Anal.}, 46(4):2743--2766, 2014.

\bibitem{NGW23}
Andrew~J. Nugent, Susana~N. Gomes, and Marie-Therese Wolfram.
\newblock On evolving network models and their influence on opinion formation.
\newblock {\em Preprint arXiv:2305.09483}, 2023.

\bibitem{PoG16}
Mason~A. Porter and James~P. Gleeson.
\newblock {\em Dynamical systems on networks}, volume~4 of {\em Frontiers in
  Applied Dynamical Systems: Reviews and Tutorials}.
\newblock Springer, Cham, 2016.
\newblock A tutorial.

\bibitem{SYT02}
Philip Seliger, Stephen~C. Young, and Lev~S. Tsimring.
\newblock Plasticity and learning in a network of coupled phase oscillators.
\newblock {\em Phys. Rev. E}, 65:041906, Mar 2002.

\bibitem{WHK22}
Dirk Witthaut, Frank Hellmann, J\"urgen Kurths, Stefan Kettemann, Hildegard
  Meyer-Ortmanns, and Marc Timme.
\newblock Collective nonlinear dynamics and self-organization in decentralized
  power grids.
\newblock {\em Rev. Mod. Phys.}, 94:015005, Feb 2022.

\end{thebibliography}
\end{document}